\documentclass{amsart}

\usepackage{amsfonts}       
\usepackage{amsthm}
\usepackage{amssymb}
\usepackage{xcolor}
\usepackage{soul}
\usepackage[pdftex]{graphicx}
\usepackage{caption}
\usepackage{float}

\newtheorem{theorem}{Theorem}[section]

\theoremstyle{definition}

\theoremstyle{remark}

\numberwithin{equation}{section}

\newcommand{\R}{\mathbb{R}}

\def\Int{\displaystyle\int}

\begin{document}

\title{Finsler-Rellich inequalities involving the distance to the boundary}

\author{G. Barbatis}
\address{Department of Mathematics, National and Kapodistrian University of Athens}
\email{gbarbatis@math.uoa.gr}

\author{M. Paschalis}
\address{Department of Mathematics, National and Kapodistrian University of Athens}
\email{mpaschal@math.uoa.gr}





\begin{abstract}
We study Rellich inequalities associated to higher-order elliptic operators in the Euclidean space. The inequalities are expressed in terms of an associated Finsler metric.
In the case of a half-space we obtain the sharp constant while for a general convex domain we obtain estimates that are better than those obtained by comparison with the polyharmonic operator.
\end{abstract}

\maketitle


\section{Introduction}
In \cite{O}, Owen proves the higher-order Rellich inequality
\begin{equation}
\int_\Omega u(x)(-\Delta)^m u(x) dx \geq A(m) \int_\Omega \frac{u^2(x)}{d^{2m}(x)}dx, \qquad  u\in C^\infty_c(\Omega),
\label{rellich}
\end{equation}
for the polyharmonic operator $(-\Delta)^m$, where $\Omega \subseteq \mathbb{R}^n$ is a convex open set, $d:\Omega \rightarrow \mathbb{R}_+$ is the Euclidean distance from the boundary of $\Omega$ and $A(m)$ is the best constant given explicitly by
$$
A(m) = \frac{(2m-1)^2(2m-3)^2 \cdots 1^2}{4^m}.
$$
This inequality has been subsequently extended and improved in various directions. In \cite{A} and for the case $2m=4$ a simple sufficient condition was given for non-convex domains so that the Rellich inequality
is valid with the sharp constant 9/16; in \cite{BT,B} sharp improvements to (\ref{rellich}) were obtained. We refer to the recent book \cite{BEL} for additional information.

While the literature for Rellich inequalities for the polyharmonic operator $(-\Delta)^m$ is substantial, there are hardly any results on Rellich inequalities involving the distance to the boundary for more general higher-order elliptic operators. This is partly due to the lack of invariance under rotations and to the (related) fact that neither the Euclidean metric nor indeed any other Riemannian metric is suitable for the study of such operators. 

Anisotropic Hardy inequalities with distance to the boundary have recently been obtained in \cite{DBG,MST}. Concerning anisotropic (non-Riemannian) Rellich inequalities, there is a growing literature on inequalities with distance to a point, see e.g. \cite{KY,KR,RSS,RS} and references therein, but we are not aware of any results involving the distance to the boundary. To our knowlegde, the best Rellich constant for $\int |\Delta u|^pdx$ is not known even in the case of a half-space.

The objective of this note is to investigate inequalities of the form
\begin{equation}
\int_\Omega u(x)Hu(x)dx \geq \kappa \int_\Omega \frac{u^2(x)}{d_H^{2m}(x)} dx
\label{kappa}
\end{equation}
where $H$ is a homogeneous elliptic differential operator of order $2m$ with real constant coefficients and
$d_H$ is a suitable Finsler distance to the boundary of $\Omega$ associated to $H$.
In particular, we will prove the following result for half-spaces which is shown to be optimal in an important class of cases.

\begin{theorem}
Let $H$ be a homogeneous elliptic operator of order $2m$ with real constant coefficients and let $\mathbf{H} \subseteq \mathbb{R}^n$ be a half-space. Then the inequality $$\int_{\mathbf{H}} u(x)Hu(x)dx \geq A(m) \int_{\mathbf{H}} \frac{u^2(x)}{d_H^{2m}(x)}dx$$ holds for all $u\in C^\infty_c(\mathbf{H})$.
\label{thm0}
\end{theorem}

Note that since the operator $H$ is not rotationally invariant, proving the inequality for the
commonly used half-space $\mathbb{R}^n_+ = \{ x\in \mathbb{R}^n : x_n > 0 \}$ does not imply the validity of the inequality for half-spaces in other directions.

In the second part of this note we investigate the case where $\Omega \subseteq \mathbb{R}^n$ is an arbitrary convex domain, and we provide a uniform (independent of the domain) lower bound for the best constant which - although most likely non-optimal - is nonetheless better than what can be achieved by simply comparing with $(-\Delta)^m$.

\section{Preliminaries}

Let $H$ be a homogeneous elliptic differential operator of order $2m$
with real constant coefficients, acting on real-valued functions on $\mathbb{R}^n$. So $H$ has the form
$$
H = (-1)^m \sum_{|\alpha| = 2m} a_\alpha D^\alpha,
$$ where $a_\alpha$ is a constant for each multi-index $\alpha$ and $D^{\alpha}=\partial_{x_1}^{\alpha_1}\ldots \partial_{x_n}^{\alpha_n}$.
The symbol of the operator $H$ is the polynomial $H:\mathbb{R}^n \rightarrow \mathbb{R}$ given by
$$
H(\xi) = \sum_{|\alpha| = 2m} a_\alpha \xi^\alpha.
$$
Setting $F_H(\xi) = H^{1/{2m}}(\xi)$ (which is positively homogeneous of order one in $\xi$), we define the associated Finsler norm $F_H^*:\mathbb{R}^n \rightarrow \mathbb{R}$ by
\begin{equation}
F_H^*(\omega) = \sup_{\xi \neq 0} \frac{\omega \cdot \xi}{F_H(\xi)}= \max_{|\xi|=1} \frac{\omega \cdot \xi}{F_H(\xi)}.
\label{eq:finsler}
\end{equation}
The Finsler distance of two points $x,x'\in \mathbb{R}^n$ is then defined as $F_H^*(x-x')$.
It is well known, see e.g. \cite{EP}, that this is the distance suitable to use when studying properties of $H$, especially so when one seeks sharp constants.
From now on we will suppress the index $H$ when there is no ambiguity and simply write $F$ for $F_H$ and $F^*$ for $F_H^*$. It is clear from the definition that
for any $\omega,\xi \in \mathbb{R}^n$ we have the inequality
\begin{equation}
\label{eq1}
H(\xi) \,  F^*(\omega)^{2m}\geq  ( \omega \cdot \xi)^{2m} .
\end{equation}
Now let $\Omega \subseteq \mathbb{R}^n$ be open with non-empty boundary and let
$d(x)$ denote the Euclidean distance of $x\in\Omega$ to $\partial \Omega$.
The Euclidean distance of a point $x\in \Omega$ to $\partial\Omega$ along the direction
$\omega \in S^{n-1}$ is given by
$$
d_\omega(x) = \inf \{ |s|: x+s\omega \notin \Omega \},
$$
and we have
$$
d(x)= \min_{\omega \in S^{n-1}} d_\omega(x).
$$
In the context of Finsler geometry, distances are scaled by the Finsler norm (\ref{eq:finsler})
along each direction, so the Finsler distance of $x$ from the boundary of $\Omega$ along the direction $\omega$ is given by
$$
d_{H,\omega}(x) = F^*(\omega) d_\omega(x).
$$
Denoting by
\begin{equation}
d_H(x) =\min\{ F^*(x-y) : y\in\partial\Omega\} \; , \quad x\in\Omega ,
\label{eq:y}
\end{equation}
the Finsler distance to the boundary we then have
$$
d_H(x) = \min_{\omega \in S^{n-1}} d_{H,\omega}(x) = \min_{\omega \in S^{n-1}}  \big(F^*(\omega)d_\omega(x) \big).
$$

\section{Finsler-Rellich inequality for half-spaces}

Let $\nu \in S^{n-1}$ and $\mathbf{H}^n_\nu = \{ x\in \mathbb{R}^n : \nu \cdot x > 0 \}$.
The Euclidean distance of $x\in \mathbf{H}^n_\nu$ from $\partial \mathbf{H}^n_\nu$
in the direction of $\omega \in S^{n-1}$ is given by
$$
d_\omega(x) = \frac{\nu \cdot x}{|\nu \cdot \omega|},
$$
and so the Finsler distance to $\partial \mathbf{H}^n_\nu$ is given by
$$
d_H(x) = \min_{\omega \in S^{n-1}} \big(F^*(\omega)d_\omega(x)\big) = \min_{\omega \in S^{n-1}} \bigg( \frac{F^*(\omega)}{|\nu \cdot \omega|} \bigg) \nu \cdot x \, ;
$$
so the minimum is achieved independently of $x$. Letting $\theta \in S^{n-1}$ be a unit vector that achieves the minimum we arrive at
\begin{equation}
d_H(x) = F^*(\theta)d_\theta(x) = \frac{ \nu\cdot x}{F^{**}(\nu)} =\frac{ d(x)}{F^{**}(\nu)} .
\label{eq:6}
\end{equation}
We are now ready to prove Theorem \ref{thm0}. We restate it as follows.
\begin{theorem}
\label{thm1}
Let $H$ be a homogeneous elliptic operator of order $2m$ with real constant coefficients. Then the inequality
\begin{equation}
\int_{\mathbf{H}^n_\nu} u(x)Hu(x)dx \geq A(m) \int_{\mathbf{H}^n_\nu} \frac{u^2(x)}{d_H^{2m}(x)}dx
\label{102}
\end{equation}
holds for any $\nu \in S^{n-1}$ and all $u\in C^{\infty}_c(\mathbf{H}^n_{\nu})$. Moreover, the constant $A(m)$ is optimal in the case where $F_H$ is a convex function.
\end{theorem}
\begin{proof}
Let $\hat{u}(\xi)$, $\xi\in \mathbb{R}^n$, denote the Fourier transform of $u$. Recalling  (\ref{eq1}),
applying Plancherel's theorem and using the one-dimensional Rellich inequality we obtain
\begin{eqnarray*}
\int_{\mathbf{H}^n_\nu} u(x)Hu(x)dx &=& \int_{\mathbb{R}^n} H(\xi) |\hat{u}(\xi)|^2 d\xi  \\
&\geq&  \frac{1}{F^*(\theta)^{2m}} \int_{\mathbb{R}^n} (\theta \cdot \xi)^{2m} |\hat{u}(\xi)|^2 d\xi \\
&=& \frac{1}{F^*(\theta)^{2m}} \int_{\mathbf{H}^n_\nu}  (\partial_\theta^{m} u(x))^2 dx \\
& \geq&  \frac{A(m)}{F^*(\theta)^{2m}} \int_{\mathbf{H}^n_\nu} \frac{u^2(x)}{d_\theta^{2m}(x)} dx \\
& =& A(m) \int_{\mathbf{H}^n_\nu} \frac{u^2(x)}{d_H^{2m}(x)}dx.
\end{eqnarray*}
To prove the optimality, we proceed as follows. For $\epsilon>0$ we consider the function $g_{\epsilon}(t) =t^{ \frac{2m-1}{2} +\epsilon}$, $t>0$. This is a sequence of minimizers for the one-dimensional
Rellich inequality of order $m$, that is
\begin{equation}
\label{eq:A}
\frac{ \Int_0^1  (g_{\epsilon}^{(m)}(t))^2 dt}{ \Int_0^1  \frac{g_{\epsilon}^2}{t^{2m}} dt}  \longrightarrow A(m) \, , \qquad \mbox{ as }\epsilon\to 0+ \, .
\end{equation}
Let $v_{\epsilon}(x)=g_{\epsilon}( x\cdot \nu)$. For any multiindex $\alpha$ with $|\alpha|=2m$ we then have
$D^{\alpha}v_{\epsilon}(x)=\nu^{\alpha} g_{\epsilon}^{(2m)}( x\cdot \nu)$ and therefore
\begin{equation}
Hv_{\epsilon}(x)=(-1)^m H(\nu) g_{\epsilon}^{(2m)}( x\cdot \nu)
\label{eq:5}
\end{equation}
We next localize $v_{\epsilon}$. We consider a function $\psi\in C^{\infty}_c(\R)$ such that $0\leq \psi\leq 1$, $\psi(t)=1$, if $|t|\leq 1/2$, $\psi(t)=0$, if $|t|\geq 1$.
Let $\pi_\nu: \mathbf{H}^n_\nu \rightarrow \partial \mathbf{H}^n_\nu $ denote the orthogonal projection from the half-space to its boundary. We define
\[
\phi(x) =\psi(\nu\cdot x)\psi (\pi_\nu(x)) \; , \qquad  u_{\epsilon}(x) =\phi(x)v_{\epsilon}( x).
\]
Then $u_{\epsilon}\in H^m_0(\mathbf{H}^n_\nu)$ and $\|u_{\epsilon}\|_{H^m_0(\mathbf{H}^n_\nu)} \to +\infty$ as $\epsilon\to 0+$. We shall estimate
$\int_{\mathbf{H}^n_\nu} u_{\epsilon} \, Hu_{\epsilon} dx$ and for this we note that when we use Leibniz rule to expand $Hu_{\epsilon} =H(\phi v_{\epsilon} )$ any term containing at least one derivative of
$\phi$ stays bounded as $\epsilon\to 0$. Setting $k=\int_{\R^{n-1}} \psi(|y|)^2dy$ and applying (\ref{eq:5}) we thus have
\begin{eqnarray}
\int_{\mathbf{H}^n_\nu} u_{\epsilon} \, Hu_{\epsilon} dx &=&  \int_{\mathbf{H}^n_\nu} \phi^2 v_{\epsilon} \, Hv_{\epsilon}dx  +O(1) \nonumber \\
&=&  k(-1)^m  H(\nu) \int_0^1  \psi^2 g_{\epsilon} \,  g^{(2m)}_{\epsilon}dt  +O(1) \nonumber \\
&=& kH(\nu) \int_0^1 (g^{(m)}_{\epsilon})^2dt  +O(1)  . \label{100}
\end{eqnarray}
On the other hand, recalling also (\ref{eq:6}) we similarly have
\begin{eqnarray}
\int_{\mathbf{H}^n_\nu} \frac{u_{\epsilon}^2(x)}{ d_H^{2m}(x)}dx &=&  F^{**}(\nu)^{2m} \int_{\mathbf{H}^n_\nu } \frac{\phi^2 v_\epsilon^2}{d^{2m}} dx  \nonumber \\
&=&  k F^{**}(\nu)^{2m}   \int_0^1 \frac{g_\epsilon^2}{t^{2m}} dt +O(1) .  \label{101}
\end{eqnarray}
From (\ref{100}), (\ref{101}) and (\ref{eq:A}) we conclude that
\begin{eqnarray*}
\frac{\Int_{\mathbf{H}^n_\nu} u_{\epsilon}(x)Hu_{\epsilon}(x)dx }{\Int_{\mathbf{H}^n_\nu} \frac{u_{\epsilon}^2(x)}{ d_H^{2m}(x)}dx}  
&=&  \bigg(\frac{F(\nu)}{F^{**} (\nu) } \bigg)^{2m}  \frac{\Int_0^1 (g^{(m)}_{\epsilon}(t))^2dt +O(1) }{\Int_0^1 \frac{g_\epsilon^2(t)}{t^{2m}} dt +O(1) } \\
&\to &  \bigg(\frac{F(\nu)}{F^{**} (\nu) } \bigg)^{2m} A(m) , \qquad \mbox{ as }\epsilon\to 0+.
\end{eqnarray*}
Since $F$ is convex, $F=F^{**}$, and optimality follows.
\end{proof}
\noindent
{\bf Remark.} It is known \cite[Section 1.6]{S} that the set $\{ \xi\in\R^n : F^{**}(\xi) \leq 1\}$ is the convex hull of the set $\{ \xi\in\R^n : F(\xi) \leq 1\}$. This shows that
$ F^{**}(\xi)\leq F(\xi)$ for all $\xi\in\R^n$ and also that there exists
a direction $\nu\in S^{n-1}$ such that $ F^{**}(\nu)= F(\nu)$. It follows
in particular that if $F$ is not convex the constant $A(m)$ is still the best possible constant for which (\ref{102}) is valid for all
$\nu\in S^{n-1}$ and all $u\in C^\infty_c(\mathbf{H}^n_{\nu})$.

\section{Convex domains}

If the symbol $H(\xi)$ of the operator $H$ satisfies
\[
\lambda |\xi|^{2m} \leq H(\xi) \leq \Lambda |\xi|^{2m}, \qquad \xi\in\R^n,
\]
then applying the polyharmonic Rellich inequality (\ref{rellich}) we obtain that
for any convex domain $\Omega\subset\R^n$ there holds
\begin{equation}
\int_{\Omega} u(x)Hu(x)dx \geq A(m) \frac{\lambda}{\Lambda} \int_{\Omega} \frac{u^2(x)}{d_H^{2m}(x)}dx  \, , \quad u\in C^{\infty}_c(\Omega).
\label{stn}
\end{equation}
In this section we adapt Davies' well known mean distance function method \cite{D}
to establish an alternative lower bound for the best Rellich constant of
(\ref{stn}).
While we have not obtained the actual constant $A(m)$, we nevertheless provide a constant which depends only on the symbol and which can be easily computed numerically in any particular case. This has been carried out at the end of the section for two monoparametric families of operators and it turns out that the constants obtained are better than those given by (\ref{stn}).

To state our result, we need some additional definitions related to the operator in question. Assuming that $H$ is an elliptic differential operator of order $2m$ as above and denoting
by $d\sigma(\omega)$ the normalized surface measure on $S^{n-1}$, we define the positive constants $\mu_H$ and $M_H$ as the best constants for the inequalities
\[
\mu_H  \,  F^{**}_H(\xi)^{2m} \leq \int_{S^{n-1}} \frac{ (\xi\cdot \omega)^{2m}}{ F^*(\omega)^{2m}}d\sigma(\omega) \leq M_H  \, H(\xi) \, , \qquad \xi\in\R^n  .
\]
With this settled, we have the following.
\begin{theorem}
\label{thm2}
Let $H$ be a homogeneous elliptic operator of order $2m$ with real constant coefficients.
Then for any open convex set $\Omega \subseteq \mathbb{R}^n$ the inequality
\begin{equation}
\int_{\Omega} u(x)Hu(x)dx \geq A(m) \frac{\mu_H}{M_H} \int_{\Omega} \frac{u^2(x)}{d_H^{2m}(x)}dx
\label{rel10}
\end{equation}
holds for all $u\in C^\infty_c(\Omega)$.
\end{theorem}
\begin{proof}
We have
\begin{eqnarray*}
\int_\Omega u(x)Hu(x)dx &=& \int_{\mathbb{R}^n} H(\xi) |\hat{u}(\xi)|^2 d\xi  \\
&\geq & \frac{1}{M_H} \int_{S^{n-1}} \frac{1}{F^*(\omega)^{2m}} \int_{\mathbb{R}^n} (\omega \cdot \xi)^{2m} |\hat{u}(\xi)|^2 d\xi \, d\sigma(\omega) \\
&=& \frac{1}{M_H} \int_{S^{n-1}} \frac{1}{F^*(\omega)^{2m}} \int_\Omega (\partial^m_\omega u(x))^2 dx \, d\sigma(\omega).
\end{eqnarray*}
We now apply the one-dimensional Rellich inequality in the direction $\omega$ to get
\begin{equation}
\int_\Omega u(x)Hu(x)dx \geq
A(m)\frac{1}{M_H} \int_\Omega u^2(x) \int_{S^{n-1}} \frac{1}{(F^*(\omega)d_\omega(x))^{2m}} d\sigma(\omega) dx.
\label{eq:est}
\end{equation}
To estimate the last integral we consider a point $x\in\Omega$ and a point $y = y(x) \in \partial \Omega$ that realizes the infimum in (\ref{eq:y}).
Let $\Pi_x$ be a supporting hyperplane at $y(x)$ and let $N=N(x)$ be the outward normal unit vector to $\Pi_x$. 
We denote by $z(\omega)=z(\omega,x)$ the intersection of $\Pi_x$
with the line $\{x+t\omega: t\in \mathbb{R}\}$. From the previous discussion, it follows that $|z(\omega)-x| \geq d_\omega(x)$ and therefore
$F^*(z(\omega)-x) \geq F^*(\omega)d_\omega(x)$ for all $x\in\Omega$ and $\omega\in S^{n-1}$.

Let $s\in\R$ be such that $z(\omega)=x+s\omega$. Since $z(\omega)$ and $y$
both belong to $\Pi_x$, $z(\omega)-y$ is perpendicular to $N$, that is
$$
(x+s\omega-y) \cdot N = 0.
$$
It follows that
$$
s=\frac{(y-x)\cdot N}{\omega \cdot N},
$$
and so
$$
z(\omega)=x+\frac{(y-x)\cdot N}{\omega \cdot N} \omega.
$$
Returning to (\ref{eq:est}), we now have
\begin{eqnarray*}
\int_{S^{n-1}} \frac{1}{(F^*(\omega)d_\omega(x))^{2m}} d\sigma(\omega) &\geq &\int_{S^{n-1}} \frac{1}{F^*(z(\omega)-x)^{2m}} d\sigma(\omega) \\
&=&\frac{1}{((y-x)\cdot N)^{2m}} \int_{S^{n-1}} \bigg( \frac{\omega \cdot N}{F^*(\omega)} \bigg)^{2m} d\sigma(\omega) \\
&\geq & \mu_H \bigg( \frac{F^{**}(N)}{(y-x)\cdot N} \bigg)^{2m} \\
&\geq& \frac{\mu_H}{F^*(y-x)^{2m}} = \frac{\mu_H}{d_H^{2m}(x)},
\end{eqnarray*}
and the proof is complete.
\end{proof}
We think of estimate (\ref{rel10}) as an explicit estimate in the sense that $\mu_H$ and $M_H$ can be computed numerically in any specific case.
In the next two examples we illustrate the estimate of Theorem \ref{thm2} and in particular
show that inequality (\ref{rel10}) is better than the one obtained from (\ref{stn}).

\

\noindent
{\bf Example 1.} Let $\beta>-1$ (for ellipticity) and
\[
H_{\beta}(\xi) =\xi_1^4 +2\beta  \xi_1^2\xi_2^2 + \xi_2^4  , \qquad \xi\in\R^2.
\]
We have
\[
\left\{ 
\begin{array}{ll}
  \frac{\beta +1}{2} |\xi|^4  \leq H_{\beta}(\xi)  \leq  |\xi|^4 , &  \mbox{ if } -1<\beta\leq 1, \\[0.2cm]
   |\xi|^4  \leq H_{\beta}(\xi)  \leq  \frac{\beta +1}{2} |\xi|^4 , & \mbox{ if } \beta\geq 1, 
\end{array}
\right.
\]
hence (\ref{stn}) gives
\[
\int_{\Omega} u(x)H_{\beta}u(x)dx \geq \frac{9}{16}c(\beta) \int_{\Omega} \frac{u^2(x)}{d_{H_{\beta}}^{2m}(x)}dx \; , \qquad u\in C^{\infty}_c(\Omega),
\]
where
\[
c(\beta)=\left\{  
\begin{array}{ll}
\frac{\beta +1}{2} , &  \mbox{ if } -1<\beta\leq 1, \\[0.1cm]
 \frac{2}{\beta +1} , & \mbox{ if } \beta\geq 1 \, .
\end{array}
\right.
\]
In Figure 1 below we have plotted the function $s(\beta) =\mu_{H_{\beta}} / M_{H_{\beta}} $ (blue line) against $c(\beta)$ (red line) and it is seen that the estimate of
Theorem \ref{thm2} is better than the one obtained from (\ref{stn}).

\

\noindent
{\bf Example 2.} Let 
\[
\hat{H}_{\beta}(\xi) =\xi_1^6 +\beta  \xi_1^4\xi_2^2+ \beta  \xi_1^2\xi_2^4 + \xi_2^6  , \qquad \xi\in\R^2.
\]
We have $\hat{H}_{\beta}(\xi)  =(\xi_1^2+\xi_2^2) [ (\xi_1^2-\xi_2^2)^2 +(\beta+1)\xi_1^2\xi_2^2]$, so we assume  $\beta>-1$ for ellipticity. We now have
\[
\left\{ 
\begin{array}{ll}
  \frac{\beta +1}{4} |\xi|^4  \leq \hat{H}_{\beta}(\xi)  \leq  |\xi|^4 , &  \mbox{ if } -1<\beta\leq 3, \\[0.2cm]
   |\xi|^4  \leq \hat{H}_{\beta}(\xi)  \leq  \frac{4}{\beta +1}|\xi|^4 , & \mbox{ if } \beta\geq 3, 
\end{array}
\right.
\]
hence (\ref{stn}) gives
\[
\int_{\Omega} u(x)\hat{H}_{\beta}u(x)dx \geq \frac{9}{16}\hat{c}(\beta) \int_{\Omega} \frac{u^2(x)}{d_{\hat{H}_{\beta}}^{2m}(x)}dx \; , \qquad u\in C^{\infty}_c(\Omega),
\]
where
\[
\hat{c}(\beta)=\left\{  
\begin{array}{ll}
\frac{\beta +1}{4} , &  \mbox{ if } -1<\beta\leq 3, \\[0.1cm]
 \frac{4}{\beta +1} , & \mbox{ if } \beta\geq 3 \, .
\end{array}
\right.
\]
In Figure 2 below we have plotted the function $\hat{s}(\beta) =\mu_{\hat{H}_{\beta}} / M_{\hat{H}_{\beta}} $ (blue line) against $\hat{c}(\beta)$ (red line).

\begin{figure*}[ht]
\centering
\captionsetup{justification=centering,margin=0.2cm}
\begin{minipage}{.5\textwidth}
  \centering
  \includegraphics[width=1\linewidth]{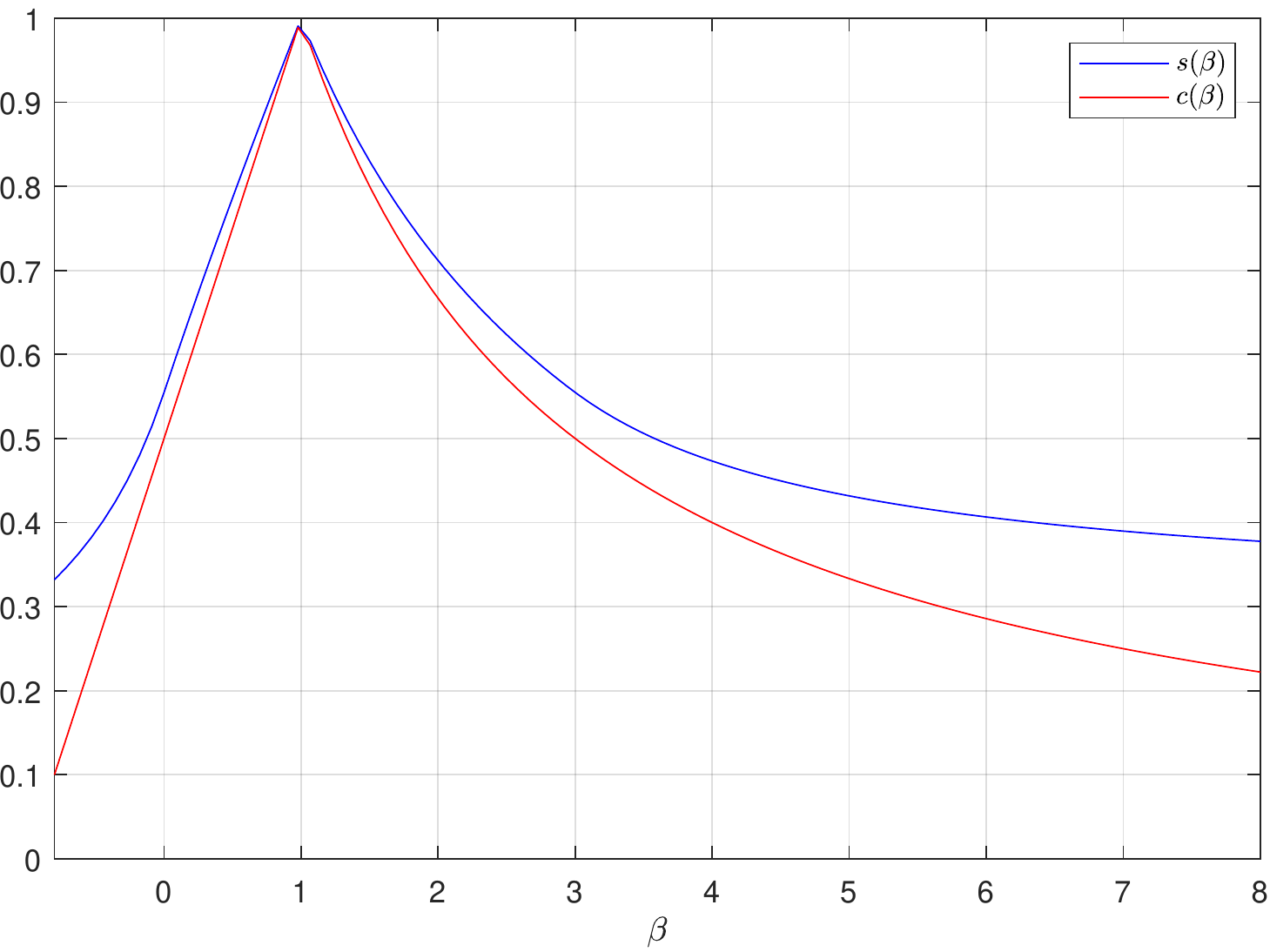}
\caption{\small Plots of $s(\beta)$ and $c(\beta)$}
\end{minipage}%
\begin{minipage}{.5\textwidth}
  \centering
 \includegraphics[width=1\linewidth]{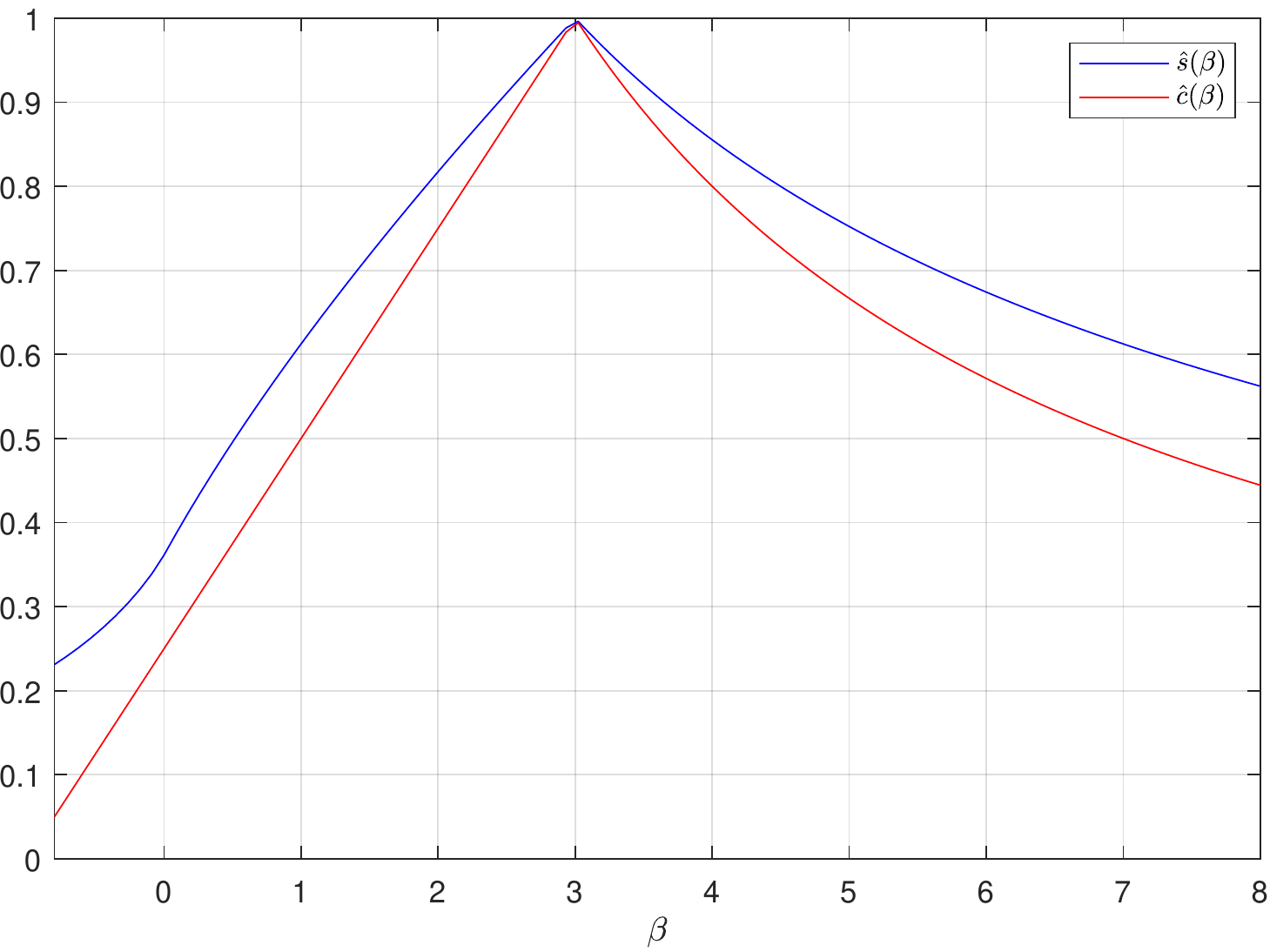}
\caption{ \small Plots of $\hat{s}(\beta)$ and $\hat{c}(\beta)$}
\end{minipage}
\end{figure*}

\noindent
{\bf Remark.} Considering Example 1, it is not difficult to prove that for any $\beta>-1$ there holds
\[
 \frac{1}{4} \max\{ 1 , \frac{2}{\beta +1} \} |\xi|^4 \leq   F^*( \xi)^4   \leq   
 \max\{ 1 , \frac{2}{\beta +1} \} |\xi|^4  \; , \quad \xi \in\R^2 \, .
\]
It then follows easily that $s(\beta)\geq 3/32$ for all $\beta>-1$. Hence not only is $s(\beta)$ larger than $c(\beta)$ but we also have that  $s(\beta)/c(\beta) \to +\infty$
as $\beta\to -1$ or $\beta\to +\infty$. Similarly, in Example 2 one has
\[
 \frac{1}{8} \max\{ 1 , \frac{2}{\beta +1} \} |\xi|^6 \leq   F^*( \xi)^6   \leq   
 \max\{ 1 , \frac{2}{\beta +1} \} |\xi|^6  \; , \quad \xi \in\R^2 \, .
\]
It follows that $\hat{s}(\beta)\geq 5/128$ and therefore  $\hat{s}(\beta)/\hat{c}(\beta) \to +\infty$
as $\beta\to -1$ or $\beta\to +\infty$.

\

\noindent
{\bf Acknowledgement.} We thank G. Kounadis for his help with the Matlab diagrams. The research of MP was supported by the Hellenic Foundation for Research and Innovation (HFRI) under the HFRI PhD Fellowship grant (Fellowship Number 1250).

\bibliographystyle{amsplain}

\end{document}